\documentclass[12pt]{article}
\usepackage{verbatim}
\usepackage{amssymb,amsmath}
\usepackage{amsthm}
\usepackage{graphicx}
\usepackage{epsfig}
\newtheorem{corollary}{Corollary}[section]
\newtheorem{lemma}[corollary]{Lemma}
\newtheorem{proposition}[corollary]{Proposition}
\newtheorem{theorem}[corollary]{Theorem}
\newcommand{\Prob} {{\mathbb P}}

\newcommand{\E}{{\mathbb E}}

\newcommand{\R}{{\mathbb{R}} }

\newcommand{\dist}{{\rm dist}}

\def\R{\mathbb{R}}

\def\d{\delta}
\def\e{\epsilon}

\def\k{\kappa}

\def\ra{\rightarrow}

\def \Im {{\rm Im}}
\def \Re {{\rm Re}}
\def \p {\partial}

\def \Half {{\mathbb H}}

\def \Disk {{\mathbb D}}

\def \diam {{\rm diam}}

\newcommand {\exc} {{\mathcal E}}

\def  \newernot{\Phi}

\begin{document}

\title{Up-to-constants bounds on the two-point Green's function for SLE curves}

\author{Gregory F. Lawler \thanks{Research supported by National
Science Foundation grant DMS-0907143.}\\
University of Chicago  \\ \\ \\
Mohammad A. Rezaei\\ Michigan State University}

\maketitle

\begin{abstract}
The Green's function for the chordal Schramm-Loewner evolution $SLE_\kappa$
for $0 < \kappa < 8$,  gives the normalized
probability of getting near points.  We give
  up-to-constant bounds for the two-point Green's function.

\end{abstract}

\section{Introduction}

The Schramm-Loewner evolution ($SLE_\kappa$)
 is a conformally invariant family of probability
measures on curves originally given by Schramm  as a candidate
for the scaling limit of lattice models in statistical
physics.  The chordal Green's function gives the 
normalized probability that the path goes through a point
and the two-point Green's functions gives the correlations
for this quantity.  While the one-point function
is known (up to an arbitrary multiplicative constant in
the definition), and the existence of the two-point
function has been established, the exact form of the two-point
function is not known.  Estimates for the two-point
function have proved to be important in analyzing fractal
properties of the $SLE$ curves, in particular the Hausdorff
dimension and the Minkowski content.  The goal of
this paper is to 
  give up-to-constant bounds valid for all pairs of points
in a domain.  It is still open to give a closed form
for the function.\\

We start by reviewing the definition of $SLE$ and giving
the relevant known results.
See \cite{Law1} for more details. 
Suppose  that $\gamma:(0,\infty) \rightarrow
\mathbb{H} =\{x+iy: y > 0\}$
is a  curve with $\gamma(0+) \in
\mathbb{R}$ and $\gamma(t) \rightarrow \infty$ as $t \rightarrow
\infty$.  Let $H_t$ be the unbounded component of
$\mathbb{H} \setminus \gamma(0,t]$.   Using the Riemann mapping
theorem, one    can see that there is a unique conformal
transformation
$          g_t: H_t \longrightarrow \mathbb{H} $
satisfying $g_t(z) - z \rightarrow 0$ as $z \rightarrow \infty$.
For any $a > 0$, it can be parametrized so that as $z \rightarrow \infty$,
\[           g_t(z) = z + \frac{a t}{z} + O(|z|^{-2}). \]
The conformal maps $g_t$ satisfy the chordal Loewner equation
\begin{equation}  \label{loew}
       \dot g_t(z) = \frac{a}{g_t(z) - U_t} , \;\;\;\;
   g_0(z) = z,
\end{equation}
where $U_t = g_t(\gamma(t))$ is a continuous real-valued
function. The {\em Schramm-Loewner
evolution} ($SLE_\kappa$) is obtained by choosing $a = 2/\kappa$ and
$U_t $ to be a standard (one-dimensional) Brownian motion. 
In this paper we will consider only $0 < \k < 8$ and let $a = 2/\kappa
> 1/4$.  We write
\[   Z_t(z) = g_t(z) - U_t . \]

For $z
\in \mathbb{H} \setminus \{0\}$, the function $t \mapsto g_t(z)$ is
well
defined up to time $T_z := \sup\{t: {\rm Im}[g_t(z)]> 0\}$.
Rohde and Schramm \cite{RS} showed
that  for $\k < 8$ the Loewner equation above
generates   a random   curve $\gamma$, which
is also called $SLE_\kappa$, and they showed in a weak sense
that the dimension of the path is 
\begin{equation}
\label{dimmy}
     d= 1 + \frac \kappa 8 .
     \end{equation}
If $\kappa \geq  8$, the curve exists but is plane filling and is not
relevant for this paper.
 If $0 < \kappa \leq 4$, the paths are simple with
$\gamma(0,\infty) \subset \Half$  while
there are double points and $\gamma(0,\infty)
\cap \R \neq \emptyset$   for $4 < \kappa < 8$.
Moreover, if $H_t$ denotes the unbounded component
of $\mathbb{H} \setminus \gamma(0,t]$, then
\[   H_t = \{z \in \mathbb{H}: T_z > t\}. \]

Their starting point to compute \eqref{dimmy}
 was to assume that
there exists a function $G$ and a constant $\hat c$ such that
\begin{equation}  \label{onedef}
 \lim_{\e \ra 0} \e^{d-2}\, \Prob\{\dist(z,\gamma)<\e\}   = \hat c \, G(z)
, \end{equation}
 where $\gamma = \gamma(0,\infty)$. 
Although  did not establish the limit, they did note that
if  such a function exists, then the conformal Markov property of
$SLE_\kappa$ implies that
\begin{equation}  \label{localmart}
    M_t(z) = |g_t'(z)|^{2-d} \, G(Z_t(z)),
\end{equation}
must be a local martingale.  From this one can determine the
only possible value of $d$ is that given in \eqref{dimmy},
 and the function $G$ must
be a multiple of
\begin{equation}  \label{greeny}
   G(z) = \Im(z)^{d-2} \, [\sin \arg(z)]^{4a-1} .
   \end{equation}
We call $G$ (with this choice of constant) The $SLE_\kappa$
Green's function. \\

In \cite{Bf} it was proved that the Hausdorff dimension
of the path is indeed $d$, and in \cite{LR1} it was established
that the $d$-dimensional Minkowski content of $\gamma[0,t]$
 is finite and nonzero.
  In \cite{Law2},
the limit 
was shown to exist  if we replace distance with the conformal
radius of $z$ in the domain $\Half \setminus \gamma$.
More recently, \cite{LR1} established the existence
of the limit as given although the value of the constant
$\hat c$ is unknown. \\
 


The two-point Green's function is defined by
\begin{equation}  \label{twodef}
\lim_{\e \ra 0, \d \ra 0}  \e^{(d-2)} \; \delta^{(d-2)} \, \Prob\{\dist(z,\gamma) <\e,\dist(w,\gamma) < \e\}
= \hat c^2 \, G(z,w).
\end{equation}
The existence of the limit with conformal radius replacing distance was 
established in \cite{LW} and the limit with distance was proved
in \cite{LR1}.    As shown in \cite{LR1},
if $\Theta(Y)$ denotes the $d$-dimensional Minkowski content
of $V \cap \gamma$, then
\[   \E[\Theta(V)] = \hat c \int_V G(z) \, dA(z) , \;\;\;\;
   \E[\Theta(V)^2] = \hat c^2\int_V \int_V G(z,w) \, dA(z)\, dA(w) , \]
   respectively. 
   Unlike the one-point case, no exact expression
has been given for $G(z,w)$.  The goal of this
paper is   we give up-to-constants
functions by proving the following theorem. 

\begin{theorem} \label{sept8.theorem}
 There exist $0 < c_1 < c_2 < \infty$ such that
  if  $z,w \in \Half$
  with $|z| \leq |w|$,
  then
 \[  c_1 \, q^{d-2}
        \, [S(w) \vee q]^{-\beta} \leq
         \frac{G(z,w)}{G(z)\,  G(w)}   \leq c_2 \,   q^{d-2}
        \, [S(w) \vee q]^{-\beta} , \]
        where
        \[  S(w) = \sin[\arg(w)] , \;\;\;\;
        q = \frac{|w-z|}{|w|} \leq 2 , \;\;\;\; \beta=   \frac{\kappa} 8 +
        \frac 8 \kappa - 2 > 0 . \]
 \end{theorem}

Two important estimates exist in the literature now.  In \cite{LW},
and implicitly in \cite{Bf} although it was not phrased in this way,
it was shown that if $V$ is a bounded domain in $\Half$ bounded
away from the real line, then
\[          G(z,w)   \asymp_V    |z-w|^{d-2} , \]
where $\asymp_V$ indicates that the implicit constant depends
on $V$.  In \cite{LZ}, it was shown that there exists $c$ such
that for all $z,w$,
\[          G(z,w) \geq c \, G(z) \, G(w) . \]

While we have defined the Green's function in terms of $SLE$ in $\Half$, it
can easily be extended to simply connected domains $D$.
 To be more precise, suppose that $D$ is a simply connected
domain and $w_1,w_2$ are distinct points in $\partial D$.  Let $F: \mathbb{H}
\rightarrow D$ be a conformal transformation of $\mathbb{H}$ onto $D$ with
$F(0) = w_1, F(\infty) = w_2$.  Then the distribution of
\[        \tilde \gamma(t) = F \circ \gamma(t) , \]
is that of $SLE_\kappa$ in $D$ from $w_1$ to $w_2$.  Although
the map $F$ is
not unique, the scaling invariance of $SLE_\kappa$ in $\mathbb{H}$ shows that the distribution is independent of the choice. 
The Green's functions $G_D(F(z);w_1,w_2), G_D(F(z),F(w);w_1,w_2)$ can be defined by conformal
covariance,
\[  G(z) = |F'(z)|^{2-d} \, G_D(F(z);w_1,w_2) , \]\[
  G(z,w) = |F'(z)|^{2-d} \, |F'(w)|^{2-d} \, G_D(F(z),F(w),w_1,w_2), \]
and the corresponding limits \eqref{onedef} and \eqref{twodef} hold.  We can
write
\[     G_D(F(z);w_1,w_2) = \Upsilon_D(F(z))^{d-2} \, S_D(F(z);w_1,w_2)^{4a-1}, \]
Here $\Upsilon_D(F(z)) = |F'(z)|^{-1}/2$ denotes $(1/2)$  times the conformal radius of
$D$ with respect to $F(z)$ and $S_D(F(z);w_1,w_2) = \sin \arg[z].$
If $\partial_1,\partial_2$ denote the two components of $\partial D \setminus
\{w_1,w_2\}$, then
\begin{equation}  \label{hmeasure}
    S_D(F(z);w_1,w_2)  \asymp \min\left\{{\rm hm}_{D}(F(z),\partial_1),
{\rm hm}_D(F(z),\partial_2) \right\}.
\end{equation}
Here, and throughout this paper, $\rm{hm}$ will denote harmonic measure; that is,
 ${\rm hm}_D(z,K)$ is the probability that a Brownian
motion starting at $z$ exits $D$ at $K$.\\

%
%
Using the Schwarz lemma and the Koebe $(1/4)$-theorem, we see that
\begin{equation} \label{Koebe}
 \frac{\Upsilon_D(z)}{2} \leq \dist(z,\partial D) \leq 2 \, \Upsilon_D(z) . 
\end{equation}

If $\gamma(t)$ is an $SLE_\kappa$ curve with transformations $g_t$
and driving function $U_t$, we write $\gamma_t = \gamma(0,t], \gamma
= \gamma_\infty$.  If $z \in \mathbb{H}$ and $t < T_z$, we let
\begin{equation}  \label{may24.1}
Z_t(z) = g_t(z) - U_t, \;\;\;\; S_t(z) = \sin \left[\arg Z_t(z)\right], \;\;\;\;
  \Upsilon_t(z) = \frac{{\rm Im}[g_t(z)]}{|g_t'(z)|}.
  \end{equation}
%
It is easy to check that if $t < T_z$, then $\Upsilon_t(z)$ as given
in \eqref{may24.1} is the same as $\Upsilon_{H_t}(z)$.  Also,
if $z \not\in \gamma$, then
  $\Upsilon(z) := \Upsilon_{T_z-}(z) = \Upsilon_D(z)$
where $D$ denotes the connected component of $\mathbb{H} \setminus
\gamma$ containing $z$.
Similarly, if $w_1,w_2$ are distinct boundary points on a simply
connected domain $D$
and $z \in D$, we define
\[            S_D(z;w_1,w_2) = \sin[\arg f(z)] , \]
where $f: D \rightarrow \mathbb{H}$ is a conformal transformation with $f(w_1) = 0,
f(w_2) = \infty$.  If $t < T_z$,  we set
$S_t(z) = S_{H_t}(z;\gamma(t),\infty)$.  If $f:D \rightarrow f(D)$ is
a conformal transformation, then it is easy to show that 
\[     S_D(z;w_1,w_2) = S_{f(D)}(f(z);f(w_1),f(w_2)). \]

 We extend the definition \eqref{greeny} as follow.
If $D$ is a simply connected domain with
distinct $w_1,w_2 \in \partial D$, we define
\[    G_D(z;w_1,w_2) = \Upsilon_D(z)^{d-2} \ S_D(z;w_1,w_2)^{4a-1} . \]
Under this definition $G(z) = G_\Half(z;0,\infty)$. 
The Green's function satisfies the conformal covariance rule
\[   G_D(z;w_1,w_2) = |f'(z)|^{2-d} \, G_{f(D)}(f(z);f(w_1),f(w_2)). \]
Note that if $t < T_z$, then
\[     M_t(z) = G_{H_t}(z;\gamma(t), \infty). \]
The local martingale $M_t(z)$ is not a martingale because
it ``blows up'' at time $t=T_z$.   If we stop it before that time, it
is actually a martingale.  To be precise, suppose that
\begin{equation} \label{tau}
        \tau = \tau_{\epsilon,z} = \inf\{t: \Upsilon_t(z).
 \leq \epsilon\}
 \end{equation}
 Then for every $\epsilon > 0$, $M_{t \wedge \tau}(z)$ is a
 martingale. 
The following is proved in \cite{Law2} (the proof there is in the upper half
plane, but it immediately extends by conformal invariance).

\begin{proposition} \label{mar14.prop1}  Suppose $\kappa < 8$,
$z \in D, w_1,w_2 \in \partial D$ and $\gamma$
is a chordal $SLE_\kappa$ path from $w_1$ to $w_2$ in $D$.  Let
$D_\infty$ denote the component of $D \setminus \gamma$ containing
$z$.  Then, as $\epsilon \downarrow 0$,
\[   \mathbb{P}\{\Upsilon_{D_\infty}(z) \leq
           \epsilon  \} \sim
  c_* \, \epsilon^{2-d} \, G_D(z) , \;\;\;\;
   c_*  = 2\,\left[\int_0^\pi  \sin^{4a}x \, dx\right]^{-1} . \]
\end{proposition}


Let us sketch the proof of the Theorem \ref{sept8.theorem}. By scaling, it suffices
to prove the theorem for $w = x_w + iy_w$ with $|w|
= 1$, in which case the conclusion can be written
as       
  \[      \frac{G(z,w)}{G(z)\,  G(w)}   \asymp |z-w|^{d-2}
         \, [y_w \vee |z-w|]^{-\beta}. \]
Here and for the reminder of this paper we write
$\asymp$ to indicate that quantities are bounded by constants
where the constants depend only on $\kappa$.
Let us give a heuristic description of this estimate to
show where this comes from.  The goal of this
paper is to justify this heuristic.
   Let $\epsilon$ be very small and let $E_z,E_w$
denote the events that $\dist(\gamma,z) < \epsilon$
and 
$\dist(\gamma,w) < \epsilon$, respectively.
\begin{itemize}
\item  The hardest part of the proof is to show that
if  $|z-w| \asymp 1$, then $E_z$ and $E_w$ are independent
events up to constants, that is,
$\Prob(E_z \cap E_w) \asymp \Prob(E_z) \, \Prob(E_w).$
\item Suppose $|z-w|$ is small and $y_w >  2 |z-w|$. 
Then $G(z) \asymp G(w) = y_w^{4a-1} \, y_w^{d-2}
 = y_w^{\beta}$.
 Let
$E'$ be the event that the path gets within distance $2|z-w|$
of $w$.  It is known that 
\[  \Prob(E') \asymp G(w) \, |z-w|^{2-d} \asymp y_w^{\beta}
 \, |z-w|^{2-d} .\]
 Given $E'$, $E_z$ and $E_w$ are conditionally independent
 up to a multiplicative constant, with
 \[   \Prob(E_z \mid E') \asymp \Prob(E_w \mid E') 
    \asymp\left[\frac \epsilon{|z-w|}\right]^{2-d} . \]
 Therefore, as $\epsilon \downarrow 0$, 
 \[ \epsilon^{2(2-d)} \, G(z,w) 
 \asymp \Prob(E_z \cap E_w) \asymp
  \Prob(E') \, \Prob(E_z \mid E')
   \, \Prob(E_w \mid E') \hspace{.5in}  \]
   \[\asymp \epsilon^{2(2-d)} \,
   y^{\beta} \, |z-w|^{d-2}  \asymp \epsilon^{2(2-d)} \,
    y^{-\beta} \, G(z) \, G(w) \, |z-w|^{d-2}. \]
    \item  Suppose $|z-w|$ is small and $y_w \leq 2 |z-w|$. 
    Again, let $E'$ be the event that the path gets within distance $2|z-w|$
of $w$.  In this case
\[   \Prob(E') \asymp  |z-w|^{4a-1} .\]
 Given $E'$, $E_z$ and $E_w$ are conditionally independent
 up to a multiplicative constant. If $\zeta
  = x_\zeta + i y_\zeta  \in \{z,w\}$, then 
  \[   \Prob(E_\zeta \mid E') \asymp \left[\frac{y_\zeta }{|z-w|}
   \right]^{4a-1} \,   \left[\frac \epsilon{y_\zeta}\right]^{2-d}
    \asymp G(\zeta) \, \epsilon^{2-d} \, |z-w|^{(d-2) + (1-4a)} . \]
  Therefore, as $\epsilon \downarrow 0$, 
 \[ \epsilon^{2(2-d)} \, G(z,w) 
 \asymp \Prob(E_z \cap E_w) \asymp
  \Prob(E') \, \Prob(E_z \mid E')
   \, \Prob(E_w \mid E') \hspace{.5in}  \]
   \[\asymp \epsilon^{2(2-d)} \,G(z) \, G(w)
|z-w|^{1-4a}  \, |z-w|^{2(d-2)}  \asymp  
    \epsilon^{2(2-d)} \,G(z) \, G(w)
|z-w|^{-\beta}  \, |z-w|^{ d-2}.\]
\end{itemize}

\section {Proof of the theorem}
We fix $0 < \kappa < 8, a = 2/\kappa$, $\beta =
   \frac{\kappa} 8 +
        \frac 8 \kappa - 2 = (4a-1) - (2-d)  > 0 . $
Let  $\gamma$ denote an $SLE_\kappa$ curve and
   \[  \gamma_t = \gamma(0,t], \;\;\;\Delta_t(z) = \dist(z,\gamma_t),
  \;\;\; \Delta(z) = \Delta_\infty(z). \]
   In \cite{LW} it is shown that for each $z,w$, there
exist $\epsilon_z,\delta_w$ such that if $\epsilon < \epsilon_z,
\delta < \delta_w$,
\begin{equation}  \label{nov17.1}
   \Prob\{ \Delta(z)\leq \epsilon\} \asymp
     G(z) \, \epsilon^{2-d} , \;\;\;\;
       \Prob\{ \Delta(w) \leq \delta\} \asymp G(w)
        \, \delta^{2-d} ,
        \end{equation}
        \begin{equation}  \label{nov17.2}
  \Prob\{ \Delta(z)\leq \epsilon,
     \Delta(w)\leq \delta\} \asymp G(z,w)
     \, \epsilon^{2-d} \, \delta^{2-d}.
     \end{equation}

    When estimating
$\Prob\{\Delta(z) \leq \epsilon\}$ there are two regimes.
The interior or bulk regime, where $\epsilon \leq \Im(z)$ can
be estimated using Proposition \ref{mar14.prop1} since in this
case $ \Delta(z)   \asymp
 \Upsilon(z)  $.  However for the boundary
regime $\epsilon > \Im(z)$,
one needs  the following estimate.

\begin{lemma}  \label{oct23}
There exists $0 < c_1 < c_2 < \infty$ such that if $ 0 < y \leq 1/4$ and
$\sigma = \inf\{t: |\gamma(t) - 1| \leq 2y\}$, then
\[  c_1 y^{4a-1} \leq \Prob\{\sigma < \infty, S_\sigma(1 + iy)
\geq 1/10\} \leq
\Prob\{\sigma < \infty\} \leq c_2 y^{4a-1}.\]
\end{lemma}

\begin{proof}  The bound $\Prob\{\sigma < \infty\} \asymp
y^{4a-1} $ can be found in a number of places.  A proof which includes a
  proof of the first inequality can be found in \cite{LZ}.  The first
  inequality is Lemma 2.10 of that paper.
\end{proof}

In particular, the lemma implies that if $\eta:(0,1) \rightarrow \Half$
is a curve with $\eta(0+),\eta(1-) \in (0,\infty)$ and $\eta = \eta(0,1)$,
then
\[    \Prob\{\gamma \cap \eta \neq \emptyset\} \leq c \, \left[\frac{\diam
(\eta)}{\dist(0,\eta)} \right]^{4a-1}. \]
One way to estimate the right-hand side is in terms of (Brownian)
excursion measure (see \cite[4.1]{LW} for definitions and similar
 estimates).  We recall that if $D$ is a simply connected domain
and $V_1,V_2$ are two arcs in $\p D$, then the excursion measure
(of the set of
excursions from $V_1$ to $V_2$ in $D$) is given by
\[  \exc_D(V_1,V_2) = \int_{V_1} \int_{V_2}  H_{\p D}(z,w) \, |dz| \, |dw|,\]
where $H_{\p D}$ denotes the boundary Poisson kernel (normal derivative
of the Green's function).  We can also write this as
\[ \exc_D(V_1,V_2) = \int_{V_1}  \p_n \phi_2(z) \, |dz| \, |dw|
 = \int_{V_2}  \p_n \phi_1(z) \, |dz| \, |dw|,\]
where $\phi_j$ is the harmonic function on $D$ with boundary value
$1_{V_j}$ and $\p_n$ denotes normal derivative.  These formulas assume
that $V_1,V_2$ are smooth; however, this quantity is a conformal
invariant so one can define this for nonsmooth boundaries.  A
standard calculation shows that if $\diam (\eta) \leq \dist(0,\eta)$,
and $H$ denotes the unbounded component of $\Half \setminus \eta$,
then
\[   \exc_{H}(\eta,(-\infty,0]) \asymp  \frac{\diam (\eta)}{\dist(0,\eta)}. \]
Suppose $\eta':(0,1) \rightarrow \Half$ is a curve in $\Half$ with $\eta'(0+)
= 0 , \eta'(1-) > 0$ that separates $\eta$ from $\infty$ in $H$.   
Let $H'$ be the bounded component of $H \setminus \eta'$.
Then monotonicity of the excursion measure implies that
\[      \exc_{H'}(\eta,\eta') \geq \exc_{H}(\eta,(-\infty,0]). \]
The upshot of this is that if we can find such an $\eta'$, then
\begin{equation}  \label{aug.1}
 \Prob\{\gamma \cap \eta \neq \emptyset\} 
   \leq c \,  \exc_{H'}(\eta,\eta')^{4a-1}.
   \end{equation}

   We will prove Theorem \ref{sept8.theorem} in a sequence
   of propositions.   We assume $|z| \leq |w|$ and let
   \[ q = |w-z|, \;\;\;\;
   \beta = (4a-1) - (2-d) = 4a + \frac{1}{4a} -2> 0. \]
   It will be useful to define a quantity that allows us
to consider  the boundary and interior cases simultaneously.
 Let
\[  \newernot_t(z)   = \Delta_t(z)^{4a-1}  \;\; \mbox{ if } \;\;
  \Delta_t(z)\geq \Im(z), \]
  \[ \newernot_t(z) =   \Im(z)^{4a-1} \,  \left[\frac{\Delta_t(z)}{\Im(z)}\right]^{2-d}
  \;\;\mbox{ if } \;\; \Delta_t(z) \leq \Im(z) , \]
     and let $\newernot(z) = \newernot_\infty(z)$.
    Note that $\newernot_0(z) = |z|^{4a-1} $, and scaling implies that
    the distribution of $\newernot(rz)$ is the same as that
    of $r^{4a-1}\newernot(z)$.  Since $4a-1 > 2-d$,
    we see that
   \begin{equation}  \label{aug.2}
     \Delta_t(z)^{4a-1}
    \leq  \Phi_t(z)   . 
    \end{equation}
    The next lemma combines the interior and boundary estimates into one
    estimate.

    \begin{lemma}  There exist $0 < c_1 < c_2 < \infty$ such
    that for all $z \in \Half$ and $0 < \epsilon \leq 1$,
 \begin{equation}  \label{sept5.1}
  c_1 \epsilon \leq \Prob\{\newernot(z) \leq \epsilon \, \newernot_0(z)\}
\leq c_2  \epsilon .
\end{equation}
   \end{lemma}

   \begin{proof}  Let $z=x+iy$.
    By scaling we may assume that $|z| = 1$ and hence $\newernot_0(z) = 1,
   S(z) = y$.  Let $\Delta = \Delta_\infty(z), \newernot = \newernot_\infty(z)$.
  Proposition \ref{mar14.prop1}
    and Lemma \ref{oct23}
    imply that
   \[  \Prob\{\Delta  \leq \epsilon \} \asymp \epsilon^{4a-1} ,\;\;\;\;
      \epsilon \geq y, \]
    \[   \Prob\{\Delta  \leq \epsilon  \} \asymp y^{4a-1} \, [\epsilon/y]^{2-d}
      ,\;\;\;\;
      \epsilon \leq y. \]
   If $\epsilon \geq y$, then
\[
   \Prob\{\newernot  \leq \epsilon^{4a-1}   \}
    = \Prob\{\Delta  \leq \epsilon  \} \asymp \epsilon^{4a-1} .\]
   If $\epsilon \leq y $, then if $u = (4a-1)/(2-d)$,
\[
     \Prob\{\newernot \leq \epsilon^{4a-1}   \}
     =   \Prob\{ y\, (\Delta /y)^{\frac{2-d}{4a-1}} \leq \epsilon   \}
    =  \Prob\{\Delta \leq y \, (\epsilon/y)^{u} \}
   \asymp y^{4a-1} \, [(\epsilon/y)^{u}]^{2-d}  = \epsilon^{4a-1}.
\]

    \end{proof}

The hardest step in estimating the
two-point Green's function is to show that if two points
are not very close to each other, then the events that the
paths get close to the two points are  independent at least up
to a multiplicative constant.  The next proposition gives
a precise version of this statement in terms of the quantity
$\newernot(z)$.

  \begin{proposition}  \label{nov17.prop1}
   There exists $c < \infty$ such that if $|z| \leq 4|w|$,
  and $0 < \epsilon_z,\epsilon_w \leq 1$, then
  \[ \Prob\{\newernot(z) \leq \epsilon_z \,\newernot_0(z) ,  \newernot(w) \leq
  \epsilon_w\, \newernot_0(w)\}  \leq c\,
  \epsilon_z \, \epsilon_w. \]
  \end{proposition}

The  proof is similar to proofs  in \cite{LW}.  The details are
somewhat technical so let us sketch the basic strategy.  The idea
is to show that if one is going to get very close to both $z$ and $w$, then
one is likely to get very close to one of them first without getting
too close to the other and then one goes to the other point.  In other words,
one does not keep going back and forth between smaller and smaller
neighborhoods of $z$ and $w$.  The way that one establishes this is
to fix  a curve ${  I}$ between $z$ and $w$ and consider excursions of the
$SLE$ paths from $I$.  What one shows is that if $\gamma$
is already very close to $z$, then it is unlikely that $\gamma$ will
get even closer to $z$ and return to ${  I}$.  There are two
different possibilities.   Suppose that  
 ${  I}_t$ is a crosscut of $H_t$ contained in ${  I}$
and $\gamma(t) \in I_t$. 
If $z$ is in the bounded component of $H_t \setminus  I_t$, then 
   $S_{H_t}(z;\gamma(t),\infty)$ is small, and the $SLE$ path
does not want to get closer to $z$.   If $z$ is in the unbounded component
of $H_t \setminus {  I}_t$, then the $SLE$ path can get closer to $z$, but
then it is unlikely to return to $I_t$.  The proof
makes this idea precise.\\

To prove Proposition \ref{nov17.prop1} we start with a lemma that  gives
     an upper bound
   for the probability that an $SLE$ path gets close to a point and
   subsequently returns to a given crosscut.  It is a generalization
   of Lemmas 4.10 and 4.11 of \cite{LW}, and we use ideas
   from those proofs. Before stating the lemma, we 
   set up some notation. Suppose $\eta:(0,1) \rightarrow \Half$
is a simple curve with $\eta(0+) = 0, \eta(1-) > 0$ and
write $\eta = \eta(0,1)$.   Let
$V_1,V_2$ denote respectively the bounded and unbounded
components of $\Half \setminus \eta$ and assume that $z
= x_z + i y_z
\in V_1, w  = x_w + i y_w\in V_2$. 
 Recall that $H_t$ is the unbounded component of $\Half \setminus \gamma_t$.
 We will let $I_t$ be a decreasing collection of subarcs of $\eta$ that
 are crosscuts of $H_t$ separating $z$ and $w$.  To be more
 specific, 
 one can show
(see \cite[Appendix A]{LW}) that there is a collection of open subarcs
$\{I_t: t < T_z \wedge T_w\}$ of $\eta$ with the following properties.

\begin{itemize}
\item  $I_0 = \eta$.
\item  $I_t \subset H_t$.  Moreover, $H_t \setminus I_t$ has
two connected components, one containing $z  $ and the other containing $w
 $.
\item  If $s < t$, then $I_t \subset I_s$.  Moreover, if
  $\gamma(s,t] \cap I_s = \emptyset$, then $I_t = I_s$.
\end{itemize}
If $\zeta \in \{z,w\}$, define
  stopping times $\sigma_{k} ,\sigma,\tau$  depending
on  $\zeta$  by
\[  \sigma_{k} = \inf\{t: \newernot_t(\zeta) = 2^{-k} \, \newernot_0(\zeta)\}, \;\;\;\;
  \sigma = \sigma_1, \]
  \[   \tau = \inf\{t \geq \sigma: \gamma(t) \in \overline {I_\sigma}
    \}= \inf\{t \geq \sigma: \gamma(t) \in   {I_\sigma}
    \}  . \]
  Here $\tau =\infty$ if $\sigma = \infty$ and the second
  equality holds with probability one. If $\tau < \infty$,   let
  \[ J  = \frac{
 \newernot_\tau(\zeta)}{\newernot_0(\zeta)} . \]

\begin{lemma}  \label{brentprop}
There exists $c < \infty $ such that under the setup
above, if $0 < \epsilon \leq 1/2$ and $\alpha =
2a - \frac 12 > 0$,
\[           \Prob\{\tau < \infty, J \leq \epsilon \}
   \leq c \,   \epsilon, \;\; \mbox{ if } \zeta = z, \]
\[     \Prob\{\tau < \infty , J \leq \epsilon \} \leq
    c \,  \epsilon \, \left[\frac{\diam(\eta)}{|w|}\right]^{\alpha} , \;\;
    \mbox{if } \zeta = w.\]
  \end{lemma}

\begin{proof}  The first inequality follows immediately from
\eqref{sept5.1}, as does the second if $|w| \leq 4 \,
\diam(\eta)$.   Therefore, using scaling, we may
 assume that $\diam(\eta) = 1, |w| \geq 4, \zeta = w$.  Let $C$
 denote the half-circle of radius $\sqrt{|w|}$ in $\Half$ centered at
 the origin.  Let $k_0$ be the largest
 integer such that $2^{-k_0} \geq S(w) = \Im(w)/|w|$.
 Let $\rho$ be the first time $t$  that $w$ is not in
 the unbounded component of $H_t \setminus C$.  
 Note that if  $\rho <
 T_w$,
 then $\gamma(\rho) \in C$.
Let
 \[     \hat J = \frac{\newernot_\rho(w)}{\newernot_0(w)}. \]
 Then, if $k$ is a positive integer and $\hat \sigma = \sigma_{k}$,
 \[   \Prob\{\tau < \infty, J \leq 2^{-k}\}
  \leq \Prob\{\hat  \sigma  < \rho  \wedge \tau, \tau < \infty\} + \sum_{j=1}^k
    \Prob\{\rho <  \hat \sigma  < \infty, 2^{-j} < \hat J \leq 2^{-j+1}\}. \]

  We will now show that
 \begin{equation}  \label{sept5.2}
 \ \Prob\{\hat  \sigma  < \rho  \wedge \tau, \tau < \infty\}
  \leq c \,2^{-k} \,  |w|^{-\alpha} .
  \end{equation}
Let $  H = H_{\hat \sigma}, I = I_{\hat \sigma}, g = g_{\hat \sigma},
U = U_{\hat \sigma}$.
 By \eqref{sept5.1},
 \[  \Prob\{\hat  \sigma  < \rho  \wedge \tau\} \leq
   \Prob\{\hat \sigma < \infty \} \leq c \, 2^{-k} . \]
   Let $H^*$ be the component of $H \setminus
 C$ containing $w$. On the event $\hat  \sigma  < \rho $,  $H^*$ is unbounded.
 Using simple connectedness of $H$, we can see that there is a
 subarc $l \in \p H^* \cap C$ that is a crosscut of $H$ and that separates
 $w$ from $I$ in $H$.  Since $l$ does not separate $w$ from $\infty$, $g(l)$
 is a crosscut of $\Half$ that does not separate $U$ from $\infty$; for ease
 let us assume that its endpoints are on $(-\infty,U]$. Since $l$
 separates $w$ from $I$,  $l$ also separates $I$ from $\infty$ in $H$.
 Therefore $g(l)$ separates $g(I)$ from $U$ and $\infty$ in $\Half$.
 We use excursion measure to estimate the probability that $\gamma[\hat \sigma, \infty)$
 returns to $I$.  The
 excursion measure between $g(I)$ and $[U,\infty)$ in $\Half \setminus
 g(I)$ is bounded above by the excursion measure between $g(I)$ and $g(l)$ in
 $\Half \setminus (g(I) \cup g(l))$ which by conformal invariance equals
 the excursion measure between $I$ and $l$ in $H \setminus(I \cup l)$.
 This in turn is bounded above by the excursion measure between
 $C$ and $\p \Disk$ in $\{\zeta \in \Half: 1 < |\zeta| < \sqrt{|w|}\}$ which
 is $O(1/\sqrt{ |w|}).$  Given this, we can use \eqref{aug.1}
 to see  that the probability that an $SLE_\kappa$
 path from $U$ to $\infty$ in $\Half$ hits $g(I)$ is $O(|w|^{-(4a-1)/2})$.  Using
 conformal invariance, we conclude that
 \[     \Prob\{\tau < \infty \mid\hat  \sigma  < \rho  \wedge \tau\}
   \leq c \, |w|^{(1-4a)/2} \]
  which gives \eqref{sept5.2}.

We noted above that if $j \leq k_0$, then
 \[  \Prob\{\rho < \hat \sigma  < \infty\} = 0 . \]
 We will now show that if $j > k_0$,
 \begin{equation}  \label{sept5.3}
  \Prob\{\rho <  \hat \sigma  < \infty, 2^{-j} < \hat J \leq 2^{-j+1}\}
   \leq c \,2^{-k}  \, 2^{-j/2} \ |w|^{-\alpha}  .  \end{equation}
   The proposition then follows by summing over $j$.
Consider the event
 \[  E_j = \{\rho <   \infty, 2^{-j} < \hat J \leq 2^{-j+1}\}. \]
  Using
  \eqref{sept5.1}, we see that
\begin{equation}  \label{sept5.4}
     \Prob(E_j) \leq c \, 2^{-j}.
     \end{equation}
 Let $H = H_\rho$.  On the event $E_j$, there is a subarc $l$ of
 $H \cap C$ that is a crosscut of $H$ with one endpoint equal
 to $\gamma(\rho)$ such that $l$ disconnects
 $w$ from $\infty$
in $H$.  Using this and the relationship between $S$ and
harmonic measure, we see that $S_\rho(w)$ is bounded above by
the probability that a Brownian motion starting at $w$ reaches
$C$ without leaving $H$.  Using \eqref{aug.2},
we see that on the event $E_j$, $\dist(w,\p H) \leq 2^{-j/(4a-1)}|w|.$
 Using the Beurling estimate, we see
that the probability a Brownian motion
starting at $w$ reaches distance $|w|/2$ from $w$  without leaving
$H$ is $O(2^{-j/2(4a-1)})$.  Given this, the probability that is reaches
$C$ without leaving $\Half$ is bounded above by $O(1/|\sqrt{|w|})$.
Therefore,  on the event $E_j$,
\[               S_\rho(w) \leq c \, 2^{-j/2(1-4a)} \, |w|^{-1/2}. \]
Using the strong Markov property and \eqref{sept5.1}, we see that
\[       \Prob\{\hat \sigma < \infty \mid E_j\} \leq c \, 2^{-j/2 }
  \, |w|^{-\alpha} \, 2^{-(k-j)}, \]
  which combined with \eqref{sept5.4} gives \eqref{sept5.3}.

 \end{proof}

  \begin{proof} [Proof of Proposition \ref{nov17.prop1}]
   By scaling, we may assume that $|z| \leq 1/2, |w| =  2$. 
   We will consider crosscuts of $H_t$ that are contained
   in the unit circle.  To be more precise,   
 we consider a decreasing collection of arcs
  $\{I_t: t < T_z \wedge T_w\}$ with the following
  properties.
  \begin{itemize}
  \item  $I_0 =    \{\zeta \in \Half: |\zeta| = 1\}.$
  \item  For each $t$, $I_t$ is a crosscut of $H_t$ that separates $z$ from
  $w$ in $H_t$.
  \item  If $t > s$, then $I_t \subset I_s$.  Moreover, if $\gamma(s,t] \cap
  \overline{I_s} = \emptyset$, then $I_t = I_s$.
  \end{itemize}
  We define a sequence of stopping times as follows.
  \[      \sigma_0 = 0 , \]
  \[       \tau_0 = \inf\{t: |\gamma(t)| = 1\} = \inf\{t: \gamma(t)
  \in \overline{I_{\sigma_0}}\}. \]
  Recursively, if $\tau_k < \infty$,
  \[     \sigma_{k+1} = \inf\left\{t > \tau_k: \newernot_t(w)  = \frac 12 \, \newernot_{\tau_k}(w)
   \mbox{ or } \newernot_t(z)  = \frac 12 \, \newernot_{\tau_k}(z) \right\},\]
   and if $\sigma_{k+1} < \infty$,
   \[    \tau_{k+1} = \inf\{t \geq \sigma_{k+1} : \gamma(t) \in \overline{I_{\sigma_{k+1}}}\}.\]
  If one of the stopping times takes on the value infinity, then all the subsequent ones
  are set equal to infinity.  If $\sigma_{k+1} < \infty$, we set $R_k = z$ if
  $\newernot_{\sigma_{k+1}}(z) =\newernot_{\tau_k}(z)/2.$  Note that in this case,
$\Delta_{\sigma_{k+1}} (z) \leq  2^{-\frac{1}{4a-1}}, $
    and
  $\newernot_{t}(w) > \newernot_{\tau_k}(w)/2$ for all $t \leq \tau_{k+1}$.
  Likewise, we set $R_k = w$ if $\newernot_{\sigma_{k+1}}(w) =\newernot_{\tau_k}(w)/2.$\\

 It follows immediately from \eqref{sept5.1} that for $r \leq 1/2$,
 \[
  \Prob\left \{    \newernot_{\tau_{0}}(z)
     \leq r \, \newernot_{0}(z) \right\} \leq c \, r , \]
   and for $r$ sufficiently small
   \[  \Prob\left \{  \newernot_{\tau_{0}}(w)
     \leq r \, \newernot_{0}(w) \right\} =0.\]

The key estimate, which we now establish,  is the following.
\begin{itemize}
\item  There exists $c,\alpha$ such that if $\tau_k < \infty, 0 <r \leq 1/2$ and $\zeta = x+iy \in \{z,w\}$,
then
\begin{equation}  \label{sept8.2}
  \Prob\left \{ \tau_{k+1} < \infty,  R_{k} = \zeta,  \newernot_{\tau_{k+1}}(\zeta)
     \leq r \, \newernot_{\tau_k}(\zeta) \mid \gamma_{\tau_k}\right\} \leq c \, r \, \newernot_{\tau_k}(\zeta)^\alpha.
     \end{equation}
  \end{itemize}

To prove, \eqref{sept8.2}, 
let $H = H_{\tau_k}, I = I_{\tau_k},\hat  g = g_{\tau_k} - U_{\tau_k},
 \hat I = \hat g(I), \hat \zeta = \hat g(\zeta), \Delta = \Delta_{\tau_k}(\zeta),
 \newernot = \newernot_{\tau_k}(\zeta), \lambda = |g'(\zeta)|$.
 Recall that $\Delta^{4a-1} \leq \newernot$.
 If $ \newernot_t(\zeta) = r \newernot$ then $|\zeta - \gamma(t)| =
 \theta \, \Delta$ where
 \[   \theta = \left[\frac{y\wedge \Delta}{\Delta} \vee r\right]^{\frac 1{4a-1}}
   \,   \left [ \frac {r\Delta} { y \wedge \Delta} \wedge 1\right]^{\frac 1{2-d}}. \]
 Note that if $r \leq 1/2$ then $\theta \leq2^{-\frac 1{4a-1}}  < 1$.\\

 Let $V$ denote the closed
  disk of radius $2^{-\frac 1{4a-1}}  \Delta$ about $\zeta$, $y_*
 = y \vee (\theta\Delta/2)$ and
 $\zeta_* = x + y_*i$. Note that $|\zeta - \zeta_*| \leq
    \theta \Delta/2 \leq 2^{-\frac 1{4a-1}}  \Delta/2$ and hence $\zeta_* \in V$.
We consider  $g$ as a conformal transformation defined on the open disk of radius
 $\Delta$ about $\zeta$; if $y < \Delta$, then we extend $g$ by Schwarz
 reflection.   By the distortion theorem, there exist $0 < c_1 < c_2 < \infty$
 such that if $\zeta_1 \in V$,
 \[                     c_1 \, \lambda \leq |\hat g'(\zeta_1)| \leq c_2 \, \lambda, \]
 \[                    c_1 \, \lambda \, |\zeta_1 - \zeta| \leq
             |\hat g(\zeta_1) - \hat \zeta| \leq c_2 \, \lambda \, |\zeta_1 - \zeta|. \]
  In particular, 
 \[               c_1 \, \lambda \, y \leq \Im \hat \zeta \leq c_2 \, \lambda \, y . \]

Note that   $\hat I$ is a crosscut of
 $\Half$ with one endpoint equal to zero.  We consider separately the cases where
 $\hat \zeta$ is in the bounded or unbounded component of $\Half \setminus \hat I$.\\

Let $E_1$ denote the event that
 $\hat \zeta$ is in the bounded component.  We claim that there
 exists $c < \infty$, such that for all $\hat \zeta' = \hat g(\zeta')
  \in \hat g( V)$,
\begin{equation}  \label{sept8.1}
        S(
       \hat \zeta') =    \frac{\Im(\hat \zeta') }{|\hat \zeta'|}
        \leq c \, \Delta^{1/2}.
           \end{equation}
  To see this, assume for ease that $\Re[\hat \zeta'] \geq 0$
  and let $\Theta = \arg \hat \zeta'$.  Then     $\Im(\hat \zeta')/|\hat \zeta'| =
  \sin \Theta  \leq \Theta$ and $\Theta/\pi$ is the probability
  that a Brownian motion starting at $\hat \zeta'$ hits $(-\infty,0]$
  before leaving $\Half$.  This is bounded above  by
    the probability that a Brownian motion starting at $\hat \zeta'$
 hits $\hat I$ before leaving $\Half$.  By conformal invariance, this
 last  probability is the
 same as the probability that a Brownian motion starting at $ \zeta'$ hits
 $I$ before leaving $H$.   The Beurling
 estimate implies that this is bounded above by
 $c \Delta^{1/2}$.  This gives \eqref{sept8.1}.   Therefore, Using \eqref{sept5.1}, there exists
 $c$ such that if $|\zeta - \gamma(t)| = \theta \Delta$, then
 
  \[  \Prob\{\newernot(\zeta) \leq r \, \newernot_{\tau_k}(\zeta), E_1 \mid
  \gamma_{\tau_k} \} \leq   c \, \sqrt{\newernot}\, r  . \]

  We now suppose that $\hat \zeta$ is in the unbounded component.
  By the same argument, for every $\hat \zeta' :=
  \hat g(\zeta') \in  \hat g( V)$, the probability
  that a Brownian motion starting at $\hat \zeta':= \hat g(\zeta')$ 
  hits $\hat I$ before leaving
  $\Half$ is bounded above by   $c\Delta^{1/2}$.  We will split
  into two subcases.  We first assume that
  \[            \Im(\hat \zeta') \leq  \Delta^{1/4} \, |\hat \zeta'|, \;\;\;\;
   \zeta' \in V. \]
  In this case, we an argue as in the previous paragraph to see that
  the probability $SLE_\kappa$ in $\Half$ hits $\hat g( V)$ is bounded above
  by  $c \, \newernot^{1/4} \, r$.  For the other case we assume that
  $\Im( \hat \zeta') \geq \Delta^{1/4} \,  |\hat \zeta'|$ for some $
  \hat\zeta' \in
  \hat g( V)$.  Using the Poisson kernel in $\Half$, we can see that the
  probability that a Brownian motion starting at $\hat \zeta'$ hits $\hat I$ before
  leaving $\Half$ is bounded below by a constant times
  \[                         \frac{\diam(\hat I)}{\Delta^{1/4} \, |\hat \zeta '|}. \] From
  this we conclude that
  \[              \diam(\hat I) \leq  c \, \Delta^{1/4}
   \, |\hat \zeta'|   . \]
  We appeal    to Lemma \ref{brentprop} to say that the probability
  that $SLE_\kappa$ in $\Half$ hits $\hat g(V)$ and then returns
  to $\hat I$ is bounded above by a constant times
  \[       r \, [\diam (\hat I)/|\hat \zeta'|]^{(4a-1)/2} \leq c\,  r \, \newernot^{1/8}.\]

  Given \eqref{sept8.2}, the remainder of the proof proceeds in the
  same way  as \cite[Section 4.4]{LW} so we omit this.

   \end{proof}

 \begin{proposition} \label{oct23.prop1}
There exist
 $0 < c_1 < c_2 < \infty$ such that if $|z| \leq |w|/4$,
 \[     c_1\,  G(z) \, G(w) \leq G(z,w) \leq c_2 \, G(z) \, G(w) .\]
 \end{proposition}

 \begin{proof}
 The bound $ G(z,w) \geq c \, G(z) \, G(w)$ was proved in \cite{LZ} so
 we need only show the other inequality. Proposition
 \ref{nov17.prop1} implies that for $\epsilon$ sufficiently small
 \[        \Prob\{\Delta(z) \leq \epsilon, \Delta(w) \leq \epsilon\}
   \leq c \, \Prob\{\Delta(z) \leq \epsilon\} \, \Prob\{\Delta(w) \leq \epsilon\} . \]
   Hence \eqref{nov17.1} and \eqref{nov17.2} imply that
   $G(z,w) \leq c \, G(z) \, G(w)$.
   \end{proof}

  The next estimate will be important even though it is not a very
  sharp bound for large $|z|,|w|$.

   \begin{proposition}  \label{sept9.prop1}
   For every $\epsilon >0$, there exists $c < \infty$ such that
   if $|z|,|w| \geq \epsilon$ and $|z-w| \geq \epsilon$, then
   \[         G(z,w) \leq c \, \Im(z)^{4a-1} \, \Im(w)^{4a-1} .\]
   \end{proposition}

   \begin{proof}  By scaling it suffices to prove the result
   when $\epsilon = 1$.  This can be done as the  proof
   of the previous proposition, so we omit
   the details.  The key step is to choose an appropriate splitting
   curve $I_0$.   We can choose $I_0$ either to be a half-circle with
   endpoints on $\R$ or a vertical line.  We choose $I_0$ so that
   $I_0$ separates $z$ and $w$ and $\dist(z,I_0), \dist(w,I_0)
   \geq 1/4$.
   \end{proof}

\begin{proof} [Proof of Theorem
   \ref{sept8.theorem}]
    By scaling, we may assume that $|w| = 1$ and hence $q = |w-z|$.
If $q \geq 1/10$,  the conclusion is
 \[            G(z,w) \asymp G(z) \, G(w).  \]
 The bound $G(z,w) \geq c \, G(z) \, G(w)$ was done in \cite{LZ}.
The other inequality
   can be deduced from Propositions \ref{oct23.prop1}
 and \ref{sept9.prop1}, respectively, for $|z| \leq 1/4$ and $|z|
 \geq 1/4$.  Here we use the fact that $G(z) \geq \Im(z)^{4a-1}$
 for $|z| \leq 1$.

 For the remainder of the proof we assume $q \leq 1/10$, and hence
 $9/10 \leq |z| \leq 1$.
 Let $z = x_z + i y_z, w = x_w + i y_w,$  and
  $\zeta = x_w + i(y_w \vee q)$. Note that 
   $G(w) \asymp
  y_w^{4a-1}, G(z) \asymp y_z^{4a-1}$.
  Let $\sigma = \inf\{t:
 |\gamma(t) - w|  = 2 q \},$ and on the event $\{\sigma < \infty\}$,
 let $h = \lambda [g_\sigma- U_\sigma]
 $ where the constant $\lambda$ is chosen so that
 $\Im[h(\zeta)] = 1$. We write
 \[ h(\zeta) =\hat \zeta = \hat x_\zeta + i, \;\;\;\;
     h(z) = \hat z =  \hat x_z + i \hat y_z, \;\;\;\; h(w) =
     \hat w =  \hat x_w + i \hat y_w.\]
      Then
\begin{eqnarray*}
 G(z,w) & = & \E\left[|g_\sigma'(z)|^{2-d} \, |g_\sigma'(w)|^{2-d}
  \, G(Z_\sigma(z), Z_\sigma(w) ); \sigma < \infty \right] \\
 & = & \E\left[|g_\sigma'(z)|^{2-d} \, |g_\sigma'(w)|^{2-d}
 \, \lambda^{2(2-d)}
  \, G(\lambda
  Z_\sigma(z), \lambda Z_\sigma(w) ); \sigma < \infty
  \right] \\ & = & \E\left[|h'(z)|^{2-d} \, |h'(w)|^{2-d} \,
     G(\hat z, \hat w); \sigma < \infty \right] .
     \end{eqnarray*}
The Koebe $(1/4)$-theorem implies that
$             |h'(\zeta)| \asymp q^{-1} . $
   Distortion estimates
 (using Schwarz reflection if $y_w  \leq 2q$) imply that
 \[    |h'(z) | \asymp |h'(w)| \asymp |h'(\zeta)| \asymp q^{-1} ,
        \]\[  |\hat z - \hat w| \asymp 1 , \]
          \[     |\hat z|, |\hat w| \geq c , \]
          \[    \hat y_z \asymp (y_z \wedge q) \, q^{-1} , \;\;\;\; \hat y_w \asymp
          (y_w \wedge q)
           \, q^{-1} . \]
These estimates hold regardless of the value of  $S(\hat \zeta)$.
  If we also know that if $S(\hat \zeta) \geq 1/10$, then
  \[     |\hat \zeta | \asymp |\hat z| \asymp | \hat w| \asymp 1.\]
   Hence, by Proposition \ref{sept9.prop1}, we see that
  \[   G(\hat z,\hat w) \leq c  \, \left[\frac{(y_z \wedge q) \, (y_w \wedge q)}
     {q^2} \right]^{4a-1}, \]
     \[     G(\hat z,\hat w) \geq c'  \, \left[\frac{(y_z \wedge q) \, (y_w \wedge q)}
     {q^2} \right]^{4a-1},\;\;\;\; \mbox{ if } S(\hat \zeta) \geq 1/10. \]
     Lemma \ref{oct23} implies that
     \[  \Prob\{\sigma < \infty\}   \asymp \Prob\{\sigma < \infty, S(\hat \zeta)
      \geq 1/10\}
      \asymp \left\{ \begin{array} {ll} y_w^{4a-1} \, (q/y_w)^{2-d} ,&   y_w \geq q \\
          q^{4a-1} , &y_w \leq q . \end{array} \right.  \]
Therefore,
\[  G(z,w) \asymp      y_w^{4a-1} \, (q/y_w)^{2-d} \,  q^{2(d-2)}\,
        \left[\frac{(y_z \wedge q) \, q}
     {q^2} \right]^{4a-1}
       , \;\;\;\; y_w \geq q, \]
     \[ G(z,w) \asymp  q^{4a-1} \, q^{2(d-2)}\,     \left[\frac{(y_z \wedge q) \, y_w}
     {q^2} \right]^{4a-1} , \;\;\;\; y_w \leq q. \]
     If $q \leq y_w \leq 2q$ we can use either  expression.
 If $y_w \leq 2q$, then $y_w \wedge q \asymp y_w,
     y_z \wedge q \asymp y_z, S(w) \vee q \asymp q$ and we can write
      \[  G(z,w)\asymp  q^{2(d-2)} \, q^{1-4a} \,   y_z^{4a-1} y_w^{4a-1}
\asymp q^{d-2} \, [S(w) \vee q]
  ^{-\beta} \, G(z) \, G(w)  . \]
  If $y_w \geq 2q$, then $y_z  \asymp y_w, y_z \wedge q \asymp q,
  S(w) \vee q \asymp y_w$, and we can write
  \[            G(z,w) \asymp y_w^{4a-1}\,q^{d-2} \,y_w^{d-2}
  = q^{d-2} \, y_w^{-\beta} \, y_w^{2(4a-1)}
   \asymp  q^{d-2} \, [S(w) \vee q]
  ^{-\beta} \, G(z) \, G(w). \]
  \end{proof}

  \section{Open problems}
  
 The obvious open problem is to determine the value of the Green's function $G(z,w)$.
 One can use the argument of Rohde and Schramm to determine a partial differential
 equation satisfied by $G$, see \cite{LW}, but it is unknown whether or not there is
 an explicit solution.\\ 
 
 One can also ask questions about the (directed) multi-point Green's function  
 $\hat{G}(z_1,...z_n)$.
 The argument in \cite{LW} can be used to show that it exists and represents
  the normalized probability of hitting $n-$point $z_1,z_2,...,z_n$ in the \emph{order} that we have them.
More precisely,
\[  \hat c^n \,
 \hat{G}(z_1,...z_n)  = \lim_{\epsilon_1,\ldots,\epsilon_n \rightarrow 0}
  \Prob\{\tau^1 < \tau^2 < \cdots < \tau^n < \infty \}, \]
 where
 \[   \tau^j = \tau^j(\epsilon_j) = \inf\{t: \Delta_t(z_j) \leq \epsilon_j \}. \]
 As a starting point, we can ask the following questions.
\begin{itemize}
\item Does there exist  $c<\infty$ such that for any $n$ and $z_1,...,z_n \in \Half$, 
\[
\hat{G}(z_1,...z_n)\leq c^n\prod_{i=1}^n|z_i-z_{i+1}|^{d-2} \;?
\]
\item  Suppose $V$ is a compact subset of $\Half$ with $\dist(0,\R) > 0$.
Is it true that
\[ \hat{G}(z_1,...z_n) \asymp_{V,n}   \prod_{i=1}^n|z_i-z_{i+1}|^{d-2} \;?\]
\end{itemize}


\begin{thebibliography}{00}



\bibitem{Bf} V. Beffara (2008).  The dimension of SLE curves, Annals of Probab.
{\bf 36},  1421-1452. 







\bibitem{Law1} G. Lawler (2005). {\em Conformally Invariant Processes
in the Plane}, Amer. Math. Soc. 


\bibitem{Law2} G. Lawler (2009). Schramm-Loewner evolution, in {\em Statistical Mechanics}, S.Sheffield and T. Spencer, ed., IAS/Park City Mathematical Series, AMS (2009), 231-295. 

\bibitem{LR1} G. Lawler, M. Rezaei,  Minkowski content and
natural parametrization for the Schramm-Loewner evolution, to appear in Annals of Probab.





 

\bibitem{LW}  G. Lawler and  B. Werness (2013).   Multi-point Green's function for SLE and an estimate of Beffara,
Annals of
Probab. {\bf 41} 1513--1555.
 

\bibitem{LZ}  G. Lawler and W. Zhou (2013).  SLE curves and natural parametrization,  Annals of
Probab. {\bf 41} 1556--1584.





\bibitem{RS} S. Rohde and O. Schramm (2005). Basic properties of
SLE, Annals of Math. {\bf 161}, 879--920.\\ 
 

 

 


 


\end{thebibliography}
\end{document}